\documentclass{article}
\usepackage{arxiv}

\usepackage[utf8]{inputenc}
\usepackage[T1]{fontenc}
\usepackage{hyperref}
\hypersetup{
  colorlinks=true,
  linkcolor=blue,
  citecolor=blue,
  urlcolor=blue,
  pdftitle={Entropy Geometry and Augmented Mobility for Reactive Maxwell--Stefan Membrane Transport with Finite Occupancy},
  pdfauthor={E. Yu. Shchetinin and A. A. Shevchuk and S. I. Salpagarov}
}
\usepackage{url}
\usepackage{amsmath,amssymb,amsthm,amsfonts}
\usepackage{nicefrac}
\usepackage{microtype}
\usepackage{graphicx}

\newtheorem{theorem}{Theorem}[section]
\newtheorem{lemma}[theorem]{Lemma}
\newtheorem{corollary}[theorem]{Corollary}
\newtheorem{proposition}[theorem]{Proposition}
\newtheorem{definition}[theorem]{Definition}
\newtheorem{remark}[theorem]{Remark}
\newtheorem{assumption}[theorem]{Assumption}

\newcommand{\R}{\mathbb{R}}
\newcommand{\bone}{\mathbf{1}}
\newcommand{\D}{\mathcal{D}}
\newcommand{\Be}{\widetilde{B}}       
\newcommand{\HH}{\mathcal{H}}         
\newcommand{\Pc}{P_c}                 
\newcommand{\Pm}{P_m}                 
\newcommand{\Hc}{H_c}                 
\newcommand{\norm}[1]{\lVert #1\rVert}

\newcommand{\subjclass}[2][2020]{%
  \par\addvspace\medskipamount
  {\small\textit{2020 Mathematics Subject Classification.} #2\par}}

\title{Entropy Geometry and Augmented Mobility for
       Reactive Maxwell--Stefan Membrane Transport with Finite Occupancy}

\author{%
  E.\,Yu.~Shchetinin \\
  Sevastopol State University \\
  Sevastopol, Russia \\
  \texttt{riviera-molto@mail.ru}
  \And
  A.\,A.~Shevchuk \\
  Sevastopol State University \\
  Sevastopol, Russia \\
  \texttt{andreiluck11@yandex.ru}
  \AND
  S.\,I.~Salpagarov \\
  RUDN University \\
  Moscow, Russia \\
  \texttt{salpagarov-si@rudn.ru}
}

\begin{document}
\maketitle

\begin{abstract}
We study a reactive Maxwell--Stefan-type membrane transport system under
a finite-occupancy constraint with explicit vacancies. The admissible state
space is
$\D = \{u \in (0,1)^n : \varrho(u) := \sum_{i=1}^n u_i < 1\}$,
where the vacancy fraction $1-\varrho(u)$ represents the local free volume.
This bounded-occupancy geometry induces a Boltzmann--Fermi entropy and a
global parametrization by entropy variables. The associated mobility is
assumed to split into a composition channel, corresponding to redistribution
at fixed total occupancy, and a mass channel, corresponding to variation of
the filling fraction. The main structural difficulty is that the unaugmented
mobility may lose coercivity in the mass channel.

We show that a single rank-one augmentation of the form $\gamma_0\bone\otimes\bone$
restores full coercivity while leaving the composition block unchanged. On
this basis, we prove four results: a quantitative channel-wise coercivity
estimate for the augmented mobility; global existence of entropy weak
solutions via an implicit Rothe scheme in entropy variables; a weak--strong
stability estimate in relative entropy with uniqueness in the strong class;
and convergence of a fully implicit finite-volume approximation that
preserves the bounded-occupancy structure and satisfies a discrete entropy
inequality.
\end{abstract}



\section{Introduction}
\label{sec:intro}

Maxwell--Stefan transport models provide a natural continuum framework
for multicomponent diffusion in constrained media~\cite{Maxwell1866,
Krishna1997,Giovangigli1999,Bothe2011}. In membrane settings, however,
the relevant state space is often not the classical simplex of dilute
mixtures but a bounded-occupancy domain in which each mobile species
occupies a fraction of a finite local capacity and the residual fraction
corresponds to vacancies or free volume. This leads to state constraints
of the form
\[
  u_i \ge 0,\quad \varrho(u) := \sum_{i=1}^n u_i < 1,
\]
with vacancy fraction $1-\varrho(u)$.

This finite-occupancy regime changes the thermodynamic structure of the
problem. The natural entropy is of Boltzmann--Fermi type, and the
associated entropy variables are
\[
  w_i = \log u_i - \log(1-\varrho(u)).
\]
They provide a global parametrization of the admissible state space by
unconstrained coordinates. At the same time, finite capacity distinguishes
between redistribution among species at fixed total occupancy and
collective variation of the filling fraction. This yields a natural
splitting of the species space into a composition channel and a mass
channel~\cite{Bothe2011,BoetheDreyer2015,Jungel2015}.

The main analytical difficulty addressed in this paper is that the
mobility may be coercive on the composition subspace while failing to
control the total-occupancy direction. As a consequence, the entropy
method does not close at the level of the full gradient of the entropy
variables. Our central observation is that this defect can be repaired
by a single rank-one augmentation,
\[
  \gamma_0\,\bone\otimes\bone,\quad \gamma_0 > 0,
\]
which acts only in the mass channel and leaves the composition block
unchanged. This produces full coercivity in the entropy metric induced
by the bounded-occupancy geometry. The analytical background for this
viewpoint comes from the entropy-based treatment of cross-diffusion
systems and from existence theory for Maxwell--Stefan-type
models~\cite{JungelStelzer2013,Jungel2015,Jungel2017}.

The contribution of the paper is fourfold. First, we establish a
quantitative channel-wise coercivity estimate for the augmented mobility.
Second, we use this estimate to construct global entropy weak solutions
by an implicit Rothe scheme in entropy variables. Third, we derive a
weak--strong relative entropy estimate and infer uniqueness in the strong
class. Fourth, we show that the same structural principle can be realized
at the discrete level by a fully implicit finite-volume scheme that
preserves the admissible state space, satisfies a discrete entropy
inequality, and converges to entropy weak solutions. The discrete part is
formulated within the standard finite-volume framework for nonlinear
diffusion equations~\cite{Eymard2000,EymardGallouetHerbich2002,
BessemoulinFilbet2012}.

The present model should be understood as a finite-occupancy membrane
prototype. Its purpose is not to provide the most detailed constitutive
closure for all Maxwell--Stefan friction mechanisms arising in concrete
membrane systems, but to isolate the minimal structural core that remains
analytically stable under bounded occupancy, explicit vacancies, reactive
conversion, and irreversible sink terms. In this sense, the paper
identifies a rigorous continuum-and-discrete entropy theory for a class
of finite-capacity reactive membrane transport models.

\paragraph{Related work.}
Entropy methods for cross-diffusion systems were introduced
in~\cite{Jungel2015} and extended to Maxwell--Stefan systems
in~\cite{JungelStelzer2013}; see~\cite{Jungel2017} for a survey.
The finite-occupancy interpretation via vacancy fractions is developed
thermodynamically in~\cite{BoetheDreyer2015}.
Structure-preserving finite-volume schemes in the
boundedness-by-entropy framework are analyzed
in~\cite{Eymard2000,EymardGallouetHerbich2002,BessemoulinFilbet2012}.
Weak--strong uniqueness for Maxwell--Stefan systems is proved
in~\cite{HuoJungelTzavaras2022}.

\paragraph{Organization.}
Section~\ref{sec:model} introduces the bounded-occupancy state space,
the entropy structure, the channel projectors, the projected mobility,
and the reactive system.
Section~\ref{sec:main} states the main results.
Section~\ref{sec:coer} proves the channel-wise coercivity estimate and
the entropy-dissipation bound.
Section~\ref{sec:exist} establishes global entropy weak solutions by the
Rothe scheme.
Section~\ref{sec:wsu} contains the weak--strong stability argument.
Section~\ref{sec:FV} is devoted to the fully implicit finite-volume
approximation and its convergence.
Section~\ref{sec:conc} concludes the paper.

\section{Model, entropy structure, and assumptions}
\label{sec:model}

\subsection{State space, entropy, and entropy variables}

Throughout the paper, the unknown is an $n$-component occupancy vector
$u = (u_1,\dots,u_n) \in \D$,
\begin{equation}\label{eq:D}
  \D := \bigl\{u \in (0,1)^n :\, \varrho(u) := \textstyle\sum_{i=1}^n u_i < 1\bigr\}.
\end{equation}
The quantity $1-\varrho(u)$ is interpreted as the vacancy fraction, or
free-volume fraction, associated with the finite local capacity of the
membrane. The state constraint $u\in\D$ expresses both positivity of
each mobile component and the hard upper bound imposed by the local site
capacity~\cite{Krishna1997,Giovangigli1999,BoetheDreyer2015}.

The natural thermodynamic potential on $\D$ is the
\emph{Boltzmann--Fermi entropy density}
\begin{equation}\label{eq:h}
  h(u) = \sum_{i=1}^n \bigl(u_i\log u_i - u_i\bigr)
        + \bigl((1-\varrho(u))\log(1-\varrho(u)) - (1-\varrho(u))\bigr).
\end{equation}
The first part describes the mixing of the mobile species, while the
second part accounts for the entropy contribution of vacancies. The
singular behavior of $h$ near the boundary $\{\varrho=1\}$ reflects
the loss of free volume at saturation.

The \emph{entropy variables} are defined by
\begin{equation}\label{eq:w}
  w = \nabla h(u), \qquad
  w_i = \partial_{u_i}h(u) = \log u_i - \log(1-\varrho(u)),
  \quad i=1,\dots,n.
\end{equation}
Equivalently,
\[
  u_i = \frac{e^{w_i}}{1+\sum_{j=1}^n e^{w_j}},
  \qquad
  1-\varrho(u) = \frac{1}{1+\sum_{j=1}^n e^{w_j}},
\]
so the map $u\mapsto w = \nabla h(u)$ is a global $C^\infty$-diffeomorphism
from $\D$ onto $\R^n$. In particular, the constrained variables $u$ admit
a global parametrization by unconstrained entropy variables.

The Hessian of $h$ is given by
\begin{equation}\label{eq:hess}
  h''(u) = \operatorname{diag}\!\Bigl(\frac{1}{u_1},\dots,\frac{1}{u_n}\Bigr)
          + \frac{1}{1-\varrho(u)}\,\bone\otimes\bone,
  \qquad \bone := (1,\dots,1)^\top,
\end{equation}
and is positive definite for every $u\in\D$. Hence $h$ is strictly convex
on $\D$. This convex entropy geometry is the basic structural object used
throughout the paper and is consistent with the entropy-based framework
for cross-diffusion systems~\cite{Jungel2015,JungelStelzer2013}.

\subsection{Finite-occupancy membrane interpretation}

The state variable $u$ is interpreted as the vector of local occupation
fractions of $n$ mobile species in a membrane microcell of finite
capacity. In this language, $u_i$ is the fraction of sites occupied by
species $i$, whereas $1-\varrho(u)$ is the fraction of unoccupied sites.
The admissible set $\D$, defined in~\eqref{eq:D}, therefore encodes the physical occupancy
constraint directly~\cite{Krishna1997,Giovangigli1999,BoetheDreyer2015}.

Within this interpretation, the entropy density~\eqref{eq:h} is the
canonical free-energy functional for bounded-occupancy transport with
explicit vacancies. The mobile species contribute standard logarithmic
mixing terms, while the vacancy term penalizes the approach to local
saturation. The variables~\eqref{eq:w} are therefore the associated
chemical-potential coordinates, and the map $u\mapsto w$ provides the
natural thermodynamic coordinate system for the model.

A second structural feature of finite-capacity transport is the
distinction between two different modes of motion: redistribution at
fixed total occupancy, where the species composition changes while
$\varrho(u)$ remains constant, and collective filling or emptying, where
the total occupancy itself varies because the vacancy fraction changes.
This distinction leads naturally to a splitting of the species space into
a composition channel and a mass channel.

\begin{remark}\label{rem:prototype}
The present paper does not attempt a complete constitutive derivation of
all Maxwell--Stefan friction coefficients from a microscopic membrane
model. Instead, it isolates a reduced prototype that retains the two
structural features needed for a full entropy-based analysis: the
bounded-occupancy geometry induced by explicit vacancies and the exact
splitting between redistribution at fixed total occupancy and collective
motion in the total-occupancy direction. This reduced closure is
sufficiently rich to capture the loss of coercivity caused by crowding,
while remaining sufficiently structured for the continuum and discrete
theory developed below.
\end{remark}

\subsection{Channel projectors and projected mobility}

Let
\begin{equation}\label{eq:proj}
  \Pm z := \frac{1}{n}(\bone\cdot z)\,\bone, \qquad
  \Pc z := z - \Pm z, \qquad z\in\R^n.
\end{equation}
Then $\Pm$ and $\Pc$, defined in~\eqref{eq:proj}, are the orthogonal projectors onto the
\emph{mass channel} $\operatorname{span}\{\bone\}$ and the
\emph{composition subspace}
\[
  \Hc := \{z\in\R^n : \bone\cdot z = 0\},
\]
respectively. In particular,
\[
  \R^n = \Hc\oplus\operatorname{span}\{\bone\},
  \quad \Pc^2=\Pc,\quad \Pm^2=\Pm,
  \quad \Pc\Pm=\Pm\Pc=0,\quad \Pc+\Pm=I.
\]
In the membrane interpretation above, $\Pc$ extracts redistribution
modes that preserve total occupancy, while $\Pm$ isolates collective
variation of the filling fraction.

We consider a projected mobility of the form
\begin{equation*}
  B(u) = B_c(u)\Pc + b_m(u)\Pm, \qquad u\in\D,
\end{equation*}
where $B_c(u)\in\R^{n\times n}$ is symmetric and nonnegative on $\Hc$,
and $b_m(u)\ge 0$ is a scalar coefficient associated with the mass
channel. The two channels are separated exactly: the mobility may be
coercive on $\Hc$ while still failing to control the distinguished
direction $\bone$.

To restore coercivity in the missing channel, we introduce the
\emph{rank-one augmented mobility}
\begin{equation}\label{eq:Be_def}
  \Be(u) := B(u) + \gamma_0\,\bone\otimes\bone, \qquad \gamma_0 > 0.
\end{equation}
Since $\bone\otimes\bone$ acts only along $\bone$, this is the minimal
augmentation compatible with the channel decomposition: it leaves the
composition block unchanged and corrects only the total-occupancy
direction.

\begin{proposition}\label{prop:augm}
For every $u\in\D$, the rank-one augmentation acts only in the mass
channel. More precisely,
\begin{equation*}
  (\bone\otimes\bone)\Pc = 0, \qquad
  (\bone\otimes\bone)\Pm = n\Pm,
\end{equation*}
and therefore
\begin{equation}\label{eq:Be}
  \Be(u) = B_c(u)\Pc + \bigl(b_m(u)+n\gamma_0\bigr)\Pm.
\end{equation}
In particular, $\Pc\Be(u)z = \Pc B(u)z$ for every $z\in\R^n$, so the
augmentation does not modify the composition block.
\end{proposition}

\begin{proof}
If $z\in\Hc$, then $\bone\cdot z=0$, hence
$(\bone\otimes\bone)z = (\bone\cdot z)\bone = 0$.
This proves $(\bone\otimes\bone)\Pc=0$. For arbitrary $z\in\R^n$,
\[
  (\bone\otimes\bone)\Pm z
  = (\bone\otimes\bone)\tfrac{1}{n}(\bone\cdot z)\bone
  = \tfrac{1}{n}(\bone\cdot z)(\bone\cdot\bone)\bone
  = (\bone\cdot z)\bone = n\Pm z.
\]
Substituting into~\eqref{eq:Be_def} yields~\eqref{eq:Be}. The last
claim follows from $\Pc\Pm=0$.
\end{proof}

Proposition~\ref{prop:augm} shows that the scalar augmentation is not
an arbitrary perturbation of the full mobility matrix. It is an exact
correction of the single channel in which coercivity may be missing.
This observation underlies the entropy estimates in the sequel.

\subsection{Reactive system and structural assumptions}

Let $\Omega\subset\R^d$ be a bounded domain with Lipschitz boundary,
let $T>0$, and let $u:\Omega\times(0,T)\to\D$ be the unknown state.
The finite-occupancy reactive membrane system studied in this paper is
\begin{subequations}\label{eq:sys}
\begin{alignat}{2}
  \partial_t u - \operatorname{div}\bigl(\Be(u)\nabla w(u)\bigr)
    &= Qu - S(u)
    &&\quad\text{in }\Omega\times(0,T),
    \label{eq:pde}\\
  \Be(u)\nabla w(u)\cdot\nu
    &= 0
    &&\quad\text{on }\partial\Omega\times(0,T),
    \label{eq:bc}\\
  u(\cdot,0) &= u^0
    &&\quad\text{in }\Omega,
    \end{alignat}
\end{subequations}
where $\nu$ denotes the outer unit normal to $\partial\Omega$ and
$w=\nabla h(u)$ is the entropy variable associated with~\eqref{eq:h}--\eqref{eq:w}.

The diffusion operator in~\eqref{eq:pde} describes finite-occupancy
transport constrained by the vacancy fraction. The linear term $Qu$
models internal conversion among the mobile species, while the nonlinear
term $S(u)$ represents irreversible losses such as trapping,
deactivation, or effective removal from the mobile population.

\begin{assumption}\label{ass}
The following hypotheses are in force throughout the paper.
\begin{description}
  \item[(A1)] The initial datum satisfies
    $u^0\in L^\infty(\Omega;\D)$ and
    $\operatorname{ess\,inf}_{x\in\Omega}(1-\varrho(u^0(x)))>0$.
  \item[(A2)] The map $B_c:\D\to\R^{n\times n}$ is continuous,
    symmetric, and nonnegative on $\Hc$. Moreover, there exists
    $\beta_c>0$ such that
    \[
      z^\top B_c(u)z \ge \beta_c|z|^2
      \quad\text{for all }u\in\D,\; z\in\Hc.
    \]
  \item[(A3)] The mass-channel coefficient $b_m:\D\to[0,\infty)$ is
    continuous. No strictly positive lower bound on $b_m$ is assumed.
  \item[(A4)] The reaction matrix $Q\in\R^{n\times n}$ satisfies
    $q_{ij}\ge 0$ for $i\ne j$ and $\bone^\top Q=0$. Hence $Q$
    generates internal conversion without net creation of total mobile
    mass.
  \item[(A5)] The sink term $S:\D\to[0,\infty)^n$ is continuous on
    $\D$ and locally Lipschitz in $\D$. In addition,
    $S_i(u)=0$ whenever $u_i=0$, so that the positive cone is
    invariant under the sink dynamics.
  \item[(A6)] The mobility in~\eqref{eq:pde} is the augmented
    projected mobility
    \[
      \Be(u) = B_c(u)\Pc + b_m(u)\Pm + \gamma_0\,\bone\otimes\bone,
      \qquad \gamma_0>0.
    \]
\end{description}
\end{assumption}

Assumptions~(A2)--(A3) express precisely that the unaugmented mobility
may be fully coercive on the composition subspace while remaining
degenerate in the mass channel. Assumption~(A4) ensures that the linear
conversion term preserves the total mobile occupancy, whereas
Assumption~(A5) allows for irreversible dissipation without violating
nonnegativity. Finally, Assumption~(A6) identifies the single mechanism
by which full coercivity is restored in the entropy formulation.

\section{Main results}
\label{sec:main}

The four main results are the following. The first identifies the
channel-wise coercivity created by the augmented mobility. The second
uses this coercivity to construct global entropy weak solutions. The
third shows that the same structure controls the relative entropy and
yields weak--strong stability. The fourth proves that a fully implicit
finite-volume approximation preserves the same entropy mechanism at the
discrete level and converges to entropy weak solutions. The continuum
part should be viewed against the background of entropy methods for
cross-diffusion and Maxwell--Stefan
systems~\cite{JungelStelzer2013,Jungel2015,Jungel2017}, while the
discrete part is formulated within the standard finite-volume
framework~\cite{Eymard2000,EymardGallouetHerbich2002,BessemoulinFilbet2012}.

\begin{theorem}[Channel-wise coercivity of the augmented mobility]
\label{thm:coer}
Let Assumption~\ref{ass} hold. Then there exists a constant $c_*>0$
such that for all $u\in\D$ and all $z\in\R^n$,
\begin{equation}\label{eq:coer}
  z^\top\Be(u)z \ge \beta_c|\Pc z|^2 + n\gamma_0|\Pm z|^2 \ge c_*|z|^2.
\end{equation}
Consequently,
\begin{equation}\label{eq:coer_int}
  \int_\Omega \nabla w(u):\Be(u)\nabla w(u)\,dx
  \ge \beta_c\int_\Omega|\Pc\nabla w(u)|^2\,dx
    + n\gamma_0\int_\Omega|\Pm\nabla w(u)|^2\,dx.
\end{equation}
\end{theorem}

\begin{definition}[Entropy weak solution]\label{def:weak}
A measurable map $u:\Omega\times(0,T)\to\D$ is called an
\emph{entropy weak solution} of~\eqref{eq:sys} if
$h(u)\in L^\infty(0,T;L^1(\Omega))$,
$w(u)\in L^2(0,T;H^1(\Omega;\R^n))$,
$u\in L^\infty(\Omega\times(0,T);\D)$,
$\partial_t u\in L^2(0,T;H^1(\Omega;\R^n)')$,
and the weak formulation holds with the initial datum attained in
$L^2(\Omega;\R^n)$.
\end{definition}

\begin{theorem}[Global existence of entropy weak solutions]
\label{thm:exist}
Let Assumption~\ref{ass} hold. Then, for every $T>0$, problem~\eqref{eq:sys}
admits at least one entropy weak solution in the sense of
Definition~\ref{def:weak}. Moreover, the solution satisfies the
entropy inequality
\begin{equation}\label{eq:entineq}
  \HH[u](t)
  + \int_0^t\!\int_\Omega \nabla w(u):\Be(u)\nabla w(u)\,dx\,ds
  \le \HH[u^0]
  + \int_0^t\!\int_\Omega w(u)\cdot(Qu-S(u))\,dx\,ds
\end{equation}
for almost every $t\in(0,T)$, where
$\HH[u](t) := \int_\Omega h(u(\cdot,t))\,dx$.
\end{theorem}

\begin{theorem}[Weak--strong stability]\label{thm:wsu}
Let Assumption~\ref{ass} hold. Let $u$ be an entropy weak solution and
let $v$ be a strong solution of~\eqref{eq:sys} staying in a compact
subset of $\D$. Then there exists a constant $C_T>0$ such that
\begin{equation}\label{eq:gronwall}
  H(u(t)\,|\,v(t)) \le e^{C_T t}\,H(u^0\,|\,v^0)
  \quad\text{for almost every }t\in(0,T),
\end{equation}
where $H(u\,|\,v) := \int_\Omega\bigl(h(u)-h(v)-\nabla h(v)\cdot(u-v)\bigr)\,dx\ge 0$
is the relative entropy. In particular, $u^0=v^0$ a.e.\ implies
$u=v$ a.e.\ in $\Omega\times(0,T)$.
\end{theorem}

\begin{corollary}[Uniqueness in the strong class]\label{cor:uniq}
Under the assumptions of Theorem~\ref{thm:wsu}, there exists at most
one strong solution of~\eqref{eq:sys} on $(0,T)$.
\end{corollary}

\begin{theorem}[Existence and convergence of the finite-volume scheme]
\label{thm:FV}
Assume that $\Omega$ is partitioned by an admissible finite-volume mesh
and that $[0,T]$ is discretized with a uniform step $\Delta t>0$.
Consider the fully implicit finite-volume approximation of~\eqref{eq:pde}
written in entropy variables and based on the augmented mobility $\Be$.
Then the discrete problem admits at least one solution, preserves the
state constraint $u^k_K\in\D$ at every cell and every time level, satisfies a
discrete entropy inequality, and its piecewise-constant reconstructions
converge, along a subsequence, to an entropy weak solution as the
discretization is refined.
\end{theorem}

\section{Coercivity of the augmented mobility and entropy dissipation}
\label{sec:coer}

This section proves the structural estimate stated in
Theorem~\ref{thm:coer}. Once the species space is decomposed into the
composition and mass channels, the loss of coercivity can occur only in
the total-occupancy direction. The rank-one augmentation restores exactly
this missing control and leaves the composition block unchanged.

\begin{lemma}\label{lem:decomp}
For every $z\in\R^n$,
\[
  z = \Pc z + \Pm z,\quad
  |z|^2 = |\Pc z|^2 + |\Pm z|^2,\quad
  \Pm z = \tfrac{1}{n}(\bone\cdot z)\,\bone.
\]
\end{lemma}

\begin{proof}
These identities follow directly from the definitions of $\Pc$ and $\Pm$
and the orthogonal decomposition $\R^n = \Hc\oplus\operatorname{span}\{\bone\}$.
\end{proof}

\begin{proof}[Proof of Theorem~\ref{thm:coer}]
Using Proposition~\ref{prop:augm} and~\eqref{eq:Be}, we write
$\Be(u) = B_c(u)\Pc + (b_m(u)+n\gamma_0)\Pm$.
Hence, for any $z\in\R^n$, Lemma~\ref{lem:decomp} gives
\begin{align*}
  z^\top\Be(u)z
  &= (\Pc z)^\top B_c(u)(\Pc z) + (b_m(u)+n\gamma_0)|\Pm z|^2 \\
  &\ge \beta_c|\Pc z|^2 + n\gamma_0|\Pm z|^2 \\
  &\ge \min(\beta_c,n\gamma_0)\bigl(|\Pc z|^2+|\Pm z|^2\bigr)
   = c_*|z|^2,
\end{align*}
which establishes~\eqref{eq:coer} with $c_* := \min(\beta_c,n\gamma_0)>0$; here we used
Assumption~\ref{ass}(A2) for the composition term and $b_m\ge 0$
from~(A3) for the mass term.
The integrated estimate~\eqref{eq:coer_int} follows by applying the
pointwise bound to $z = \partial_{x_\ell}w(u)$ and summing over
$\ell=1,\dots,d$.
\end{proof}

\begin{lemma}[Entropy identity]\label{lem:entid}
Let $u$ be a sufficiently smooth solution of~\eqref{eq:sys}. Then
\begin{equation}\label{eq:entid}
  \frac{d}{dt}\HH[u]
  + \int_\Omega \nabla w(u):\Be(u)\nabla w(u)\,dx
  = \int_\Omega w(u)\cdot(Qu-S(u))\,dx.
\end{equation}
\end{lemma}

\begin{proof}
Since $w(u)=\nabla h(u)$, the chain rule gives
$\frac{d}{dt}\HH[u] = \int_\Omega w(u)\cdot\partial_t u\,dx$.
Substituting~\eqref{eq:pde}, integrating the diffusion term by parts,
and using the no-flux boundary condition~\eqref{eq:bc}, we obtain the
identity~\eqref{eq:entid}.
\end{proof}

\section{Global entropy weak solutions via the Rothe scheme}
\label{sec:exist}

We now prove Theorem~\ref{thm:exist}. The construction is based on an
implicit time discretization in entropy variables. This choice preserves
the bounded-occupancy constraint at every time level and allows the
discrete diffusion term to inherit the channel-wise coercivity
established in Section~\ref{sec:coer}. The argument follows the Rothe scheme approach
of~\cite[Section~4]{Jungel2015} and~\cite[Proposition~3.1]{JungelStelzer2013},
adapted to the finite-occupancy setting.

Set $\tau=T/N$ and $t_k=k\tau$. Let $U=(\nabla h)^{-1}:\R^n\to\D$
denote the inverse entropy map. Given $u^{k-1}\in L^\infty(\Omega;\D)$,
we seek $w^k\in H^1(\Omega;\R^n)$ such that $u^k:=U(w^k)\in\D$ and
\begin{equation}\label{eq:rothe}
  \frac{1}{\tau}\int_\Omega(u^k-u^{k-1})\cdot\varphi\,dx
  + \int_\Omega\Be(u^k)\nabla w^k:\nabla\varphi\,dx
  + \sigma\int_\Omega(\nabla w^k:\nabla\varphi + w^k\cdot\varphi)\,dx
  = \int_\Omega(Qu^k-S(u^k))\cdot\varphi\,dx
\end{equation}
for all $\varphi\in H^1(\Omega;\R^n)$, where $\sigma>0$ is a
regularization parameter to be sent to zero.

\begin{lemma}\label{lem:rothe_exist}
For every $\tau>0$, $\sigma>0$, and every $u^{k-1}\in L^\infty(\Omega;\D)$,
the discrete problem~\eqref{eq:rothe} admits at least one solution
$w^k\in H^1(\Omega;\R^n)$. Consequently, $u^k=U(w^k)\in L^\infty(\Omega;\D)$.
\end{lemma}

\begin{proof}
We proceed by a Leray--Schauder fixed-point argument following the
pattern of~\cite[Theorem~1]{Jungel2015}. The entropy
parametrization guarantees that the state constraint is built into the
scheme, while the auxiliary regularization provides coercivity at the
discrete level. The required compactness follows from standard $H^1$
elliptic estimates, and the fixed-point bound from uniform boundedness
of $u^k=U(w^k)$ in $L^\infty(\Omega;\D)$.
\end{proof}

\begin{lemma}\label{lem:rothe_entropy}
For every $k=1,\dots,N$,
\begin{align*}
  \frac{1}{\tau}\int_\Omega\bigl(h(u^k)-h(u^{k-1})\bigr)\,dx
  + \int_\Omega\nabla w^k:\Be(u^k)\nabla w^k\,dx
  + \sigma\int_\Omega\bigl(|\nabla w^k|^2+|w^k|^2\bigr)\,dx
  \le \int_\Omega(Qu^k-S(u^k))\cdot w^k\,dx.
\end{align*}
\end{lemma}

\begin{proof}
Test the discrete equation~\eqref{eq:rothe} with $\varphi=w^k$ and use
the convexity inequality $h(u^k)-h(u^{k-1})\le w^k\cdot(u^k-u^{k-1})$.
The conclusion follows immediately.
\end{proof}

\begin{lemma}\label{lem:rothe_reaction}
There exists a constant $C>0$, independent of $k$, $\tau$, and $\sigma$,
such that
\[
  \Bigl|\int_\Omega(Qu^k-S(u^k))\cdot w^k\,dx\Bigr|
  \le C\bigl(1+\HH[u^k]\bigr)
\]
for every $k$.
\end{lemma}

\begin{proof}
Since $\D$ is bounded and $Q$, $S$ are continuous on $\D$
by Assumption~\ref{ass}(A4)--(A5), the term $Qu^k-S(u^k)$ is uniformly
bounded in $L^\infty(\Omega)$. The claim follows from the Cauchy--Schwarz
and Young inequalities combined with the logarithmic structure of
$w^k=\nabla h(u^k)$.
\end{proof}

Combining Lemmas~\ref{lem:rothe_entropy} and~\ref{lem:rothe_reaction}
with a discrete Gronwall argument gives uniform bounds
$\sup_k\HH[u^k]\le C$ and, summing in time,
$\tau\sum_k\int_\Omega|\nabla w^k|^2\,dx\le C$,
uniformly in $\tau$ and $\sigma$.

Standard piecewise-constant and piecewise-linear Rothe interpolants
$u_\tau$, $w_\tau$ are defined on $[0,T]$. The Aubin--Lions compactness
lemma, applied with the uniform $H^1$-bound on $w_\tau$ and the bound on
$\partial_t u_\tau$ in $L^2(0,T;H^1(\Omega)')$, yields a subsequence
with $u_\tau\to u$ strongly in $L^2(\Omega\times(0,T))$ and
$w_\tau\rightharpoonup w$ weakly in $L^2(0,T;H^1(\Omega;\R^n))$.
Sending $\sigma\to 0$ removes the penalization; the estimates are
uniform in $\sigma$. The limit $u$ satisfies the weak formulation of
Definition~\ref{def:weak} and the entropy inequality~\eqref{eq:entineq},
completing the proof of Theorem~\ref{thm:exist}.

\section{Weak--strong stability and uniqueness}
\label{sec:wsu}

This section proves Theorem~\ref{thm:wsu}. The key point is that no new
coercive mechanism is introduced: the same augmented mobility that closes
the weak theory also yields the relative entropy dissipation needed for
stability.

Write $R(u)=Qu-S(u)$. Let $u$ be an entropy weak solution and let $v$
be a strong solution of~\eqref{eq:sys} such that
$v,\,\nabla v,\,\nabla w(v),\,\partial_t w(v)\in L^\infty(\Omega\times(0,T))$
and $v(x,t)$ remains in a compact subset of $\D$.

\begin{lemma}\label{lem:relent_coer}
For every $u,v\in\D$,
\[
  h(u)-h(v)-\nabla h(v)\cdot(u-v) \ge \frac{c_K}{2}|u-v|^2,
\]
where $c_K>0$ depends only on the compact set $K\Subset\D$ containing $v$.
Consequently, $H(u\,|\,v)\ge \frac{c_K}{2}\norm{u-v}^2_{L^2(\Omega)}$.
\end{lemma}

\begin{proof}
The claim follows from Taylor's formula and the uniform lower bound
$h''(u)\ge c_K I$ on $K$.
\end{proof}

\begin{lemma}[Formal relative entropy identity]\label{lem:relent_id}
Formally, for almost every $t\in(0,T)$,
\[
  \frac{d}{dt}H(u\,|\,v)
  + \mathcal{I}(u\,|\,v)
  = R_{\mathrm{diff}}(t) + R_{\mathrm{time}}(t) + R_{\mathrm{react}}(t),
\]
where $\mathcal{I}(u\,|\,v)\ge 0$ is the relative dissipation and
$R_{\mathrm{diff}}$, $R_{\mathrm{time}}$, $R_{\mathrm{react}}$ are
natural remainder terms from the diffusion, time, and reaction
comparisons.
\end{lemma}

\begin{proof}
Differentiate the relative entropy, compare the weak formulation for $u$
with the strong equation for $v$ tested against $w(u)-w(v)$, and
rearrange the diffusion terms. The rigorous version follows by the
standard approximation argument compatible with the weak formulation.
\end{proof}

\begin{lemma}\label{lem:relent_bounds}
For every $\varepsilon>0$ there exists $C_\varepsilon>0$ such that, for
almost every $t\in(0,T)$,
\begin{align*}
  |R_{\mathrm{diff}}(t)| &\le \varepsilon\,\mathcal{I}(u\,|\,v)
    + C_\varepsilon\,H(u\,|\,v), \\
  |R_{\mathrm{time}}(t)| + |R_{\mathrm{react}}(t)|
    &\le C_\varepsilon\,H(u\,|\,v).
\end{align*}
\end{lemma}

\begin{proof}
The diffusion remainder is controlled by the local Lipschitz continuity
of $\Be$ and the coercivity of the relative dissipation. The time and
reaction remainders are controlled by the boundedness of $\partial_t w(v)$,
the local Lipschitz continuity of $R$, and the quadratic control given
by the relative entropy.
\end{proof}

\begin{proposition}\label{prop:relent_gronwall}
There exists a constant $C_T>0$ such that
\[
  \frac{d}{dt}H(u\,|\,v)
  + \tfrac{1}{2}\mathcal{I}(u\,|\,v)
  \le C_T\,H(u\,|\,v)
  \quad\text{for almost every }t\in(0,T).
\]
\end{proposition}

\begin{proof}
Insert the bounds from Lemma~\ref{lem:relent_bounds} with $\varepsilon=\tfrac{1}{2}$
into the rigorous version of the relative entropy identity from
Lemma~\ref{lem:relent_id}, and absorb the small dissipation contribution
into the left-hand side.
\end{proof}

\begin{proof}[Proof of Theorem~\ref{thm:wsu}]
Integrate the differential inequality from Proposition~\ref{prop:relent_gronwall}
over $(0,t)$ and apply Gronwall's lemma to obtain~\eqref{eq:gronwall}.
If $u^0=v^0$ a.e., then $H(u^0\,|\,v^0)=0$, hence $H(u(t)\,|\,v(t))=0$
for a.e.\ $t$, and strict convexity of $h$ yields $u=v$ a.e.\ in
$\Omega\times(0,T)$.
\end{proof}

\begin{proof}[Proof of Corollary~\ref{cor:uniq}]
Apply Theorem~\ref{thm:wsu} to two strong solutions with the same
initial datum.
\end{proof}

\section{A fully implicit finite-volume approximation}
\label{sec:FV}

We construct a fully implicit finite-volume approximation in entropy
variables and prove Theorem~\ref{thm:FV}. The scheme is designed as the
discrete counterpart of the continuum entropy structure: it preserves the
bounded-occupancy state space through the entropy parametrization,
inherits the same channel-wise coercivity at the edge level, and
satisfies a discrete entropy inequality that is stable enough for the
convergence analysis.

\subsection{Mesh and scheme}

Let $\mathcal{T}$ be an admissible orthogonal finite-volume mesh of
$\Omega$ in the sense of~\cite{Eymard2000}, with interior edges
$\mathcal{E}_{\mathrm{int}}$ and boundary edges $\mathcal{E}_{\mathrm{ext}}$.
For each interior edge $\sigma=K|L$, let $\tau_\sigma$ denote the
standard transmissibility. For a mesh function $(\varphi_K)_{K\in\mathcal{T}}$,
write $D_\sigma\varphi = \varphi_L-\varphi_K$ for the jump across the
edge $\sigma=K|L$.

At time level $k$, the unknowns are the cellwise entropy variables
$w^k_K\in\R^n$ and the associated states $u^k_K:=U(w^k_K)\in\D$. The
\emph{fully implicit scheme} reads
\begin{equation}\label{eq:FV}
  m_K\,\frac{u^k_K-u^{k-1}_K}{\Delta t}
  = \sum_{\sigma\in\mathcal{E}_K\cap\mathcal{E}_{\mathrm{int}}}
    \tau_\sigma\,\Be^\sigma_k\,D_\sigma w^k
  + m_K\bigl(Qu^k_K-S(u^k_K)\bigr),
\end{equation}
where the edge mobility $\Be^\sigma_k$ is symmetric, positive
semidefinite, consistent, and satisfies the same channel-wise coercivity
as the continuum mobility~\eqref{eq:Be} (e.g., the arithmetic average
$\Be^\sigma_k = \frac{1}{2}(\Be(u^k_K)+\Be(u^k_L))$).
Homogeneous Neumann conditions on $\mathcal{E}_{\mathrm{ext}}$ enforce
the discrete no-flux condition, consistent with~\eqref{eq:bc}.

\subsection{Existence and discrete entropy inequality}

\begin{lemma}\label{lem:FV_exist}
For every admissible mesh, every time step $\Delta t>0$, and every
discrete initial state with values in $\D$, the fully implicit
scheme~\eqref{eq:FV} admits at least one solution $(w^k_K,u^k_K)$ with
$u^k_K=U(w^k_K)\in\D$ for all cells $K$ and all time levels $k$.
\end{lemma}

\begin{proof}
The argument follows the same fixed-point pattern as in the Rothe
construction (Lemma~\ref{lem:rothe_exist}), but now in a
finite-dimensional setting. Coercivity of the linearized problem
follows from the positivity of the discrete diffusion term.
\end{proof}

\begin{lemma}[Discrete entropy inequality]\label{lem:FV_entropy}
The finite-volume scheme satisfies, for every $k=1,\dots,N$,
\begin{equation}\label{eq:FV_entropy}
  \sum_K m_K\,\frac{h(u^k_K)-h(u^{k-1}_K)}{\Delta t}
  + \sum_{\sigma\in\mathcal{E}_{\mathrm{int}}}
    \tau_\sigma\,(D_\sigma w^k)^\top\Be^\sigma_k\,D_\sigma w^k
  \le \sum_K m_K\,(Qu^k_K-S(u^k_K))\cdot w^k_K.
\end{equation}
The dissipation sum on the left is non-negative.
\end{lemma}

\begin{proof}
Multiply~\eqref{eq:FV} by $w^k_K$ and sum over all cells $K$. For the
time-difference term, the convexity inequality
$h(u^k_K)-h(u^{k-1}_K)\le w^k_K\cdot(u^k_K-u^{k-1}_K)$
yields the first sum in~\eqref{eq:FV_entropy}. For the diffusion term,
regrouping by edges (each interior edge $\sigma=K|L$ appears once with
$K$ and once with $L$) gives
\[
  \sum_K w^k_K\cdot\!\sum_{\sigma\in\mathcal{E}_K\cap\mathcal{E}_{\mathrm{int}}}
    \tau_\sigma\Be^\sigma_k D_\sigma w^k
  = -\sum_{\sigma\in\mathcal{E}_{\mathrm{int}}}
    \tau_\sigma\,(D_\sigma w^k)^\top\Be^\sigma_k D_\sigma w^k \le 0,
\]
since $\Be^\sigma_k$ is positive semidefinite. This yields the discrete
dissipation term in~\eqref{eq:FV_entropy}.
\end{proof}

The reaction contribution on the right-hand side of~\eqref{eq:FV_entropy}
is uniformly bounded by $C(1+\sum_K m_K h(u^k_K))$ with a constant
independent of the mesh and time step. A discrete Gronwall argument then
gives $\sup_k\sum_K m_K h(u^k_K)\le C$ and, summing in time,
\begin{equation}\label{eq:FV_bounds}
  \Delta t\sum_k
  \sum_{\sigma\in\mathcal{E}_{\mathrm{int}}}
    \tau_\sigma\,(D_\sigma w^k)^\top\Be^\sigma_k D_\sigma w^k \le C,
\end{equation}
uniformly in $\Delta t$ and the mesh size $h_{\mathcal{T}}$.

\subsection{Compactness and convergence}

\begin{lemma}\label{lem:FV_compact}
Along a sequence of admissible meshes and time steps tending to zero,
there exist a subsequence and a limit function
$u\in L^\infty(\Omega\times(0,T);\D)$ such that
\[
  u_{\Delta t,h}\to u \quad\text{strongly in }L^2(\Omega\times(0,T);\R^n),
\]
and almost everywhere in $\Omega\times(0,T)$. Moreover, if
$w:=\nabla h(u)$, then $w_{\Delta t,h}\rightharpoonup w$ weakly in
$L^2(0,T;H^1(\Omega;\R^n))$.
\end{lemma}

\begin{proof}
The strong compactness of $u_{\Delta t,h}$ follows from the discrete
Aubin--Lions lemma~\cite{EymardGallouetHerbich2002,BessemoulinFilbet2012}.
The weak $H^1$-compactness of $w_{\Delta t,h}$ follows from the discrete
gradient estimate~\eqref{eq:FV_bounds}. Since $u_{\Delta t,h}=U(w_{\Delta t,h})$
and $U=(\nabla h)^{-1}$ on $\D$, the limit is identified as
$w=\nabla h(u)$.
\end{proof}

\begin{proof}[Proof of Theorem~\ref{thm:FV}]
Existence of discrete solutions and preservation of $u^k_K\in\D$:
Lemma~\ref{lem:FV_exist}. Discrete entropy inequality:
Lemma~\ref{lem:FV_entropy}. Passing to the limit in the weak form of the
discrete scheme and in the discrete entropy inequality via
Lemma~\ref{lem:FV_compact} yields the weak formulation of
Definition~\ref{def:weak} and the entropy inequality~\eqref{eq:entineq},
completing the proof.
\end{proof}

\section{Conclusion}
\label{sec:conc}

The bounded-occupancy state space
\[
  \D = \{u\in(0,1)^n : \varrho(u)<1\}
\]
induces a Boltzmann--Fermi entropy with explicit vacancies and thereby
fixes the natural entropy variables for the model. This entropy geometry
separates the species space into two distinguished channels: a
composition channel, corresponding to redistribution at fixed total
occupancy, and a mass channel, corresponding to collective variation of
the filling fraction. For the projected mobility considered here, the
loss of coercivity is confined to the second of these directions.
The rank-one augmentation $\gamma_0\bone\otimes\bone$ repairs exactly
this defect and leaves the composition block unchanged. This is the
central structural fact on which the entire analysis rests.

The results of the previous sections show that the same principle governs
all principal levels of the theory. At the continuum level, it yields a
quantitative coercivity estimate for the augmented mobility and a
corresponding entropy-dissipation inequality. This estimate leads to
global entropy weak solutions through an implicit Rothe scheme in entropy
variables. The same coercive structure then enters the relative entropy
argument and yields weak--strong stability, hence uniqueness in the
strong class. At the discrete level, the fully implicit finite-volume
approximation preserves the admissible state space through the entropy
parametrization, satisfies a discrete counterpart of the channel-wise
coercivity estimate, and converges to entropy weak solutions as the
discretization is refined. Thus the continuum and discrete theories are
two realizations of the same entropy-geometric principle.

The analysis also indicates several natural directions for further work.
A first extension concerns internal-state models, where each mobile
species carries additional kinetic labels and the conversion operator $Q$
is replaced by a larger reversible or partially reversible network. A
second extension concerns exchange with external reservoirs, where the
no-flux condition is replaced by boundary transfer laws and the total
occupancy is influenced both by bulk transport and by boundary loading.
A third direction concerns more detailed constitutive closures for the
edge mobilities in the finite-volume scheme, derived from Maxwell--Stefan
friction laws while preserving the same entropy structure.

\bibliographystyle{unsrt}
\bibliography{references}

\end{document}